\documentclass[a4paper,12pt,makeidx]{article}
\usepackage{graphicx} 
\usepackage[utf8]{inputenc}
\usepackage{amsmath,amssymb}
\usepackage{mathrsfs}
\usepackage{tcolorbox}
\usepackage{amsthm}
\usepackage{chngcntr}
\counterwithout{figure}{section}

\newtheorem{theo}{Theorem}

\newtheorem{claim}{Claim}
\newtheorem{lemma}{Lemma}

\title{Checking the admissibility of odd-vertex pairings is hard}
\author{Florian H\"{o}rsch}

\begin{document}

\maketitle

\begin{abstract}
Nash-Williams proved that every graph has a well-balanced orientation. A key ingredient in his proof is admissible odd-vertex pairings. We show that for two slightly different definitions of admissible odd-vertex pairings, deciding whether a given odd-vertex pairing is admissible is co-NP-complete. This resolves a question of Frank. We also show that deciding whether a given graph has an orientation that satisfies arbitrary local arc-connectivity requirements is NP-complete.
\end{abstract}

\section{Introduction}
This article proves some negative results which are related to the strong orientation theorem of Nash-Williams.
\medskip

Our graphs are undirected unless specified otherwise. Let $G=(V,E)$ be a graph. For some disjoint $X,Y \subseteq V$, we use $d_G(X,Y)$ for the number of edges that are incident to one vertex in $X$ and one vertex in $Y$. We use $d_G(X)$ for $d_G(X,V-X)$. For some integer $k$,  we say that $X$ is {\it $k$-edge-connected} if $d_G(X)\geq k$ for all nonempty $X \subset V$. We abbreviate $1$-edge-connected to {\it connected}. A {\it connected component} of $G$ is a maximal connected subgraph. For some subgraph $G'$ of $G$, we denote by $G'[X]$ the subgraph of $G'$ induced by $X \cap V(G')$. For a single vertex $v$, we use $d_G(v)$ for $d_G(\{v\})$ and call this number the {\it degree} of $v$. We call $G$ {\it eulerian} if the degree of every vertex in $V$ is even. For $s,t \in V$ and $X \subseteq V$, we say that $X$ is an {\it $s\bar{t}$-set} if $s \in X$ and $t \in V-X$. We use $\lambda_G(s,t)$ for the minimum of $d_G(X)$ over all $s\bar{t}$-sets $X$. By the undirected edge version of Menger's theorem \cite{m}, this is the same as the maximum size of a set of edge-disjoint $st$-paths in $G$. For some nonempty $X \subset V$, we use $R_G(X)$ for $\max\{2\lfloor\frac{\lambda_G(s,t)}{2}
\rfloor:X\text{ is an }s\bar{t}\text{-set}\}$. We define $R_G(\emptyset)=R_G(V)=0$. For two graphs $G_1=(V,E_1)$ and $G_2=(V,E_2)$ on the same vertex set $V$, we use $G_1+G_2$ for $(V,E_1\cup E_2)$.

Let $D=(V,A)$ be a directed graph. For some $X \subseteq V$, we use  $d_D^-(X)$ for the number of arcs in $A$ entering $X$ and  $d_D^+(X)$ for $d_D^-(V-X)$. For a single vertex $v$, we use $d^-_D(v)$ and $d^+_D(v)$ for $d_D^-(\{v\})$ and $d_D^+(\{v\})$, respectively. We call $D$ {\it eulerian} if $d^-_D(v)=d_D^+(v)$ for all $v \in V$. For $s,t \in V$, we use  $\lambda_D(s,t)$ for the minimum of $d_D^+(X)$ over all $s\bar{t}$-sets $X$. By the directed arc version of Menger's theorem \cite{m}, this is the same as the maximum size of a set of arc-disjoint $st$-paths in $D$.  For two directed graphs $D_1=(V,A_1)$ and $D_2=(V,A_2)$ on the same vertex set $V$, we use $D_1+D_2$ for $(V,A_1\cup A_2)$. A directed graph $\vec{G}$ that is obtained from a graph $G=(V,E)$ by choosing an orientation for each of its edges is called an {\it orientation} of $G$. The orientation $\vec{G}$ is called {\it well-balanced} if $\lambda_{\vec{G}}(s,t)\geq \lfloor\frac{\lambda_G(s,t)}{2}\rfloor$ for all $s,t \in V$.
\medskip

In 1960, Nash-Williams proved the following celebrated theorem on well-balanced orientations \cite{NW}. 
\begin{theo}\label{well}
Every graph has a well-balanced orientation.
\end{theo}

The key ingredient in the proof of Theorem \ref{well} is the consideration of a new graph $F$ on $V$ such that $F$ is a perfect matching on the vertices in $V$ that are of odd degree in $G$. We call such a graph an {\it odd-vertex pairing} of $G$. Observe that if $F$ is an odd-vertex pairing of $G$, then $G+F$ is eulerian. Nash-Williams proves the existence of an odd-vertex pairing $F$ such that for every eulerian orientation $\vec{G}+\vec{F}$ of $G+F$, the restricted orientation $\vec{G}$ is a well-balanced orientation of $G$. We call an odd-vertex pairing $F$ with this property {\it orientation-admissible}. 

Actually, Nash-Williams proves the existence of an odd-vertex pairing with a somewhat stronger property: the odd-vertex pairings he finds satisfy the cut condition $d_G(X)-d_F(X)\geq R_G(X)$ for all $X \subseteq V$. We call such an odd-vertex pairing {\it cut-admissible}. It is easy to prove that every cut-admissible odd-vertex pairing is orientation-admissible. On the other hand, not every orientation-admissible odd-vertex pairing is cut-admissible. An example can be found in Figure \ref{draw1}.

 \begin{figure}[h]
        \centerline{\includegraphics[width=.2\textwidth]{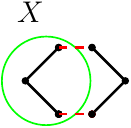}}
        \caption{The edges of $G$ are marked in solid and those of $F$ are marked in dashed. The set $X$ shows that $F$ is not cut-admissible but $F$ is trivially orientation-admissible.}
\label{draw1}
    \end{figure}

The main difficulty in the proof of Theorem \ref{well} is to show that for every graph, there is a cut-admissible odd-vertex pairing. This part of the proof is quite involved.

Kir\'aly and Szigeti use the existence of an orientation-admissible pairing to prove the existence of well-balanced orientations with some extra properties \cite{ks}. Nevertheless, most algorithmic considerations related to well-balanced orientations remain hard to deal with due to the difficulty of the proof of Theorem \ref{well}. In \cite{bikks}, Bern\'ath et al. provide a collection of negative results for questions concerning well-balanced orientations with extra properties.

This naturally raises the following question which is asked by Frank in \cite{book} as Research Problem 9.8.1. For a given odd-vertex pairing, can its admissibility properties be checked efficiently? The purpose of this work is to give a negative answer to this question. More formally, we consider the following two problems:

\bigskip
\noindent  CUT-ADMISSIBILITY ({\bf CA}):

\smallskip
\noindent Instance: A graph $G$ and an odd-vertex pairing $F$ of $G$.

\smallskip
\noindent Question: Is $F$ cut-admissible in $G$?

\bigskip
\noindent  ORIENTATION-ADMISSIBILITY ({\bf OA}):

\smallskip
\noindent Instance: A graph $G$ and an odd-vertex pairing $F$ of $G$.

\smallskip
\noindent Question: Is $F$ orientation-admissible in $G$?
\bigskip

While it is not clear whether CA and OA are in $NP$, they can easily be seen to be in co-NP. As our main results, we prove the following two theorems.

\begin{theo}\label{ca}
CA is co-NP-complete.
\end{theo}

\begin{theo}\label{oa}
OA is co-NP-complete.
\end{theo}

In the last part of this article, we consider another problem on graph orientation. Given a graph $G$, we aim to find an orientation of $G$ that meets arbitrary local arc-connectivity requirements. Formally, we consider the following problem:

\bigskip
\noindent LOCAL ARC-CONNECTIVITY ORIENTATION ({\bf LACO}):

\smallskip
\noindent Instance: A graph $G$ and a requirement function $r:V^2\rightarrow \mathbb{Z}_{\geq 0}$.

\smallskip
\noindent Question: Is there an orientation $\vec{G}$ of $G$ such that $\lambda_{\vec{G}}(u,v)\geq r(u,v)$ for all $u,v \in V^2$?

We were surprised not to find any previous work on the algorithmic tractability of this problem. By a reduction using one of the negative results in \cite{bikks}, we fill this gap.

\begin{theo}\label{loc}LACO is NP-complete.
\end{theo}
While the proof of Theorems \ref{ca} and \ref{oa} is slightly involved, the proof of Theorem \ref{loc} is quite simple.
\medskip

In Section \ref{prem}, we give some preparatory results for the proof of Theorems \ref{ca} and \ref{oa}. In Section \ref{red}, we give a reduction that serves as a proof for both Theorem \ref{ca} and Theorem \ref{oa}. Finally, in Section \ref{local}, we prove Theorem \ref{loc}.

\section{Preliminaries}\label{prem}
In this section, we collect some preliminary results we need in our reduction.

\subsection{A modified MAXCUT problem}

The unweighted MAXCUT problem can be formulated as follows:

\bigskip
\noindent {\bf MAXCUT}:

\smallskip
\noindent Instance: A graph $H=(V,E)$ and a positive integer $k$.

\smallskip
\noindent Question: Is there some $X \subseteq V$ such that $d_H(X) > k$?
\bigskip

 A proof of the following theorem can be found in \cite{gj}.

\begin{theo}
MAXCUT is NP-hard.
\end{theo}
For our reduction in Section \ref{red}, we need a slightly adapted version of MAXCUT.

\bigskip
\noindent  ADAPTED MAXCUT({\bf AMAXCUT}):

\smallskip
\noindent Instance: A graph $H=(V,E)$ such that $|E|\geq 6$ is even and $d_H(v)$ is even for every $v \in V$ and an even integer $k$.

\smallskip
\noindent Question: Is there some $X \subseteq V$ such that $d_H(X) > k$?

\begin{lemma}\label{amax}
AMAXCUT is NP-hard.
\end{lemma}
\begin{proof}
We show this by a reduction from MAXCUT. Let $(H=(V,E),k)$ be an instance of MAXCUT. We may obviously suppose that $|E|\geq 3$. Let $H'=(V,E')$ be the graph which is obtained from $H$ by replacing every edge of $E$ by $2$ parallel copies of itself. Observe that $|E'|=2|E|\geq 6$ is even and $d_{H'}(v)=2d_H(v)$ is even for all $v \in V$. Further, for every $X \subseteq V$, we have $d_{H'}(X)=2d_H(X)$. This yields that $(H,k)$ is a positive instance of MAXCUT if and only if $(H',2k)$ is a positive instance of AMAXCUT.


\end{proof}
\subsection{Augmented $(\alpha,\beta)$-grids}
In this subsection, we introduce a class of grid-like graphs which will be used as  a gadget in our reduction. A {\it grid} is a graph on ground set $\{1,\ldots,\mu\}\times \{1,\ldots,\nu\}$ for some positive integers $\mu,\nu$ where two vertices $(i_1,j_1)$ and $(i_2,j_2)$ are adjacent if $|i_1-i_2|+|j_1-j_2|=1$. For some $i \in \{1,\ldots, \mu\}$, we call $\{(i,1),\ldots, (i,\nu)\}$ the {\it row} $i$. Similarly, for some $j \in \{1,\ldots, \nu\}$, we call $\{(1,j),\ldots, (\mu,j)\}$ the {\it column} $j$.

In order to define augmented $(\alpha,\beta)$-grids for an odd integer  $\alpha \geq 3$  and an integer $\beta \geq 2$, we first consider a grid with $\alpha \beta$ rows and $\frac{\alpha +1}{2}$ columns. Now, for some $1 \leq \gamma \leq \beta$, let $L_{\gamma}=\{ l_1,\ldots,l_{\gamma}\}=\{(\alpha,1),(2 \alpha,1),\ldots,(\gamma \alpha, 1)\}$ and $P_{\gamma}=\{p_1,\ldots,p_{\gamma}\}=\{(\alpha,\frac{\alpha +1}{2}),(2 \alpha,\frac{\alpha +1}{2}),\ldots,(\gamma \alpha, \frac{\alpha +1}{2})\}$.
 We use $L$ for $L_{\beta}$ and $P$ for $P_{\beta}$. We now create the {\it augmented  $(\alpha,\beta)$-grid} $W$ by adding an edge from $(1,j)$ to $(\alpha \beta,j)$ for all $j=1,\ldots, \frac{\alpha +1}{2}$ and by adding parallel edges in the columns $1$ and $\frac{\alpha +1}{2}$ in a way that none of them is incident to a vertex in $L \cup P$ and that every vertex in $V(W)-(L\cup P)$ has degree 4 in $W$. Observe that this is possible because both $\alpha-1$ and $\alpha +1$ are even. An example can be found in Figure \ref{draw2}.

 \begin{figure}[h]
        \centerline{\includegraphics[width=.2\textwidth]{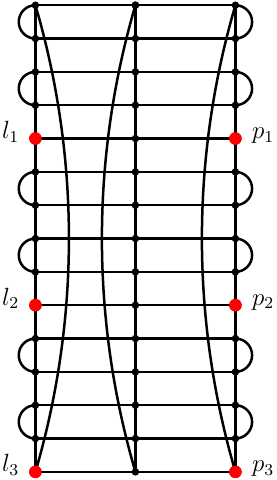}}
        \caption{An augmented $(5,3)$-grid.}\label{draw2}
    \end{figure}

 Later, when $W$ is not clear from the context, we use $L(W)$ for the set $L$ etc. We now collect some properties of augmented $(\alpha,\beta)$-grids.

\begin{lemma}\label{34conn}
Let $W=(V,E)$ be an augmented $(\alpha,\beta)$-grid for some odd integer $\alpha \geq 3$ and some integer $\beta \geq 2$. Then $W$ is $3$-edge-connected and if $d_W(X)=3$ for some nonempty $X \subset V$, then $X=\{v\}$ or $X=V-\{v\}$ for some $v\in L(W)\cup P(W)$.
\end{lemma}
\begin{proof}
Let $\emptyset \subset X \subset V$ such that $d_W(X)\leq3$. Observe that every row that intersects both $X$ and $V-X$ contributes at least $1$ to $d_W(X)$ and every column that intersects both $X$ and $V-X$ contributes at least $2$ to $d_W(X)$. It follows that one of $X$ or $V-X$ is contained in one row and in one column. We obtain that $|X|=1$ or $|V-X|=1$ and so the statement follows by construction. 
\end{proof}
\begin{lemma}\label{big}
Let $W=(V,E)$ be an augmented $(\alpha,\beta)$-grid for some odd integer $\alpha \geq 3$ and some integer $\beta \geq 2$. Further, let $X \subseteq V$ such that both $W[X]$ and $W[V-X]$ have a connected component containing at least two vertices of $L(W)\cup P(W)$. Then $d_W(X)>\alpha$.
\end{lemma}
\begin{proof}

Suppose for the sake of a contradiction that there is some $X \subseteq V$ such that both $W[X]$ and $W[V-X]$ have a connected component containing at least two vertices of $L(W)\cup P(W)$ and $d_W(X)\leq \alpha$. We choose $X$ with this property so that the total number of connected components of $W[X]$ and $W[V-X]$ is minimized. First suppose that $W[X]$ is disconnected. It follows from the assumption that $W[X]$ has a connected component $C$ such that $W[X]-C$ has a connected component containing at least two vertices in $L(W)\cup P(W)$. Let $X'=X-V(C)$. We obtain $d_W(X')\leq d_W(X) \leq \alpha$, a contradiction to the minimal choice of $X$. It follows that $W[X]$ is connected. Similarly, $W[V-X]$ is connected.

 If every column contains an element of $X$ and an element of $V-X$, each column contributes 2 to $d_W(X)$ and so $d_W(X)\geq 2\frac{\alpha +1}{2}>\alpha$. We may hence suppose by symmetry  that there is a column that is completely contained in $X$ and that there are two vertices $l_{i_1},l_{i_2}\in (V-X)\cap L$. Observe that every path from $l_{i_1}$ to $l_{i_2}$ intersects at least $|i_1-i_2|\alpha+1>\alpha$ rows. Each of these rows contributes 1 to $d_W(X)$, so $d_W(X)>\alpha$. 


\end{proof}
\subsection{Eulerian orientations}
For the proof of the co-NP completeness of OA, we need the following result on eulerian orientations which can be found in \cite{ff}.
\begin{theo}\label{ff}
Let $G,F$ be graphs on the same vertex set $V$ such that $G+F$ is an eulerian graph and let $\vec{F}$ be an orientation of $F$. Then there is an orientation $\vec{G}$ of $G$ such that $\vec{G}+\vec{F}$ is eulerian if and only if $d_G(X)\geq d_{\vec{F}}^+(X)-d_{\vec{F}}^-(X)$ for all $X \subseteq V$.
\end{theo}
\section{The reduction for admissibility}\label{red}

This section is dedicated to giving a reduction proving that CA and OA are co-NP-complete. In a first step, we reduce AMAXCUT to a problem which is somewhat similar to CA but has a more local cut condition. Next, we modify this construction to obtain a reduction for CA. Finally we show that the obtained instance is positive for OA if and only if it is positive for CA.
\subsection{The intermediate cut problem}
Let $(H=(V_H,E_H),k)$ be an instance of AMAXCUT. We abbreviate $|V_H|$ and $|E_H|$ to $n$ and $m$, respectively. Let $M=mn-k$. We now create a graph  $G_1=(V_1,E_1)$ with $V_1=V_H \cup \{q,s,t\}$ where  $q,s$ and $t$ are 3 new vertices. Let $E_1$ consist of $M$ edges from $q$ to $s$, $m$ edges from $s$ to every $v \in V_H$ and  $m$ edges from $t$ to every $v \in V_H$. A schematic drawing of $G_1$ can be found in Figure \ref{draw3}.

 \begin{figure}[h]
        \centerline{\includegraphics[width=.5\textwidth]{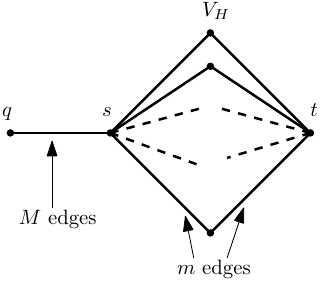}}
        \caption{A schematic drawing of $G_1$.}\label{draw3}
    \end{figure}
\begin{lemma}\label{int}
There is some $q\bar{t}$-set $X\subseteq V_1$ such that $d_{G_1}(X)-d_H(X\cap V_H)<M$ if and only if $(H,k)$ is a positive instance of AMAXCUT.
\end{lemma}
\begin{proof}
First suppose that $(H,k)$ is a positive instance of AMAXCUT, so there is some $X \subseteq V_H$ such that $d_H(X)>k$. Let $X'=\{q,s\}\cup X$. Observe that $X'$ is a $q\bar{t}$-set and $d_{G_1}(X')=mn$. This yields $d_{G_1}(X')-d_H(X'\cap V_H)=d_{G_1}(X')-d_H(X)<M$.

Now suppose that there is some $q\bar{t}$-set $X\subseteq V_1$ such that $d_{G_1}(X)-d_H(X\cap V_H)<M$. 
\begin{claim}\label{strich}
$s \in X$.
\end{claim}
\begin{proof}
Suppose otherwise. If $X=\{q\}$, then $d_{G_1}(X)-d_H(X \cap V_H)=M-0\nless M$, a contradiction. We may hence suppose that $X$ contains some $v \in V_H$. It follows from $d_H(X\cap V_H)\leq m$ and construction that $d_{G_1}(X)-d_H(X \cap V_H)\geq d_{G_1}(q,s)+d_{G_1}(v,t)-m=M+m-m\nless M$, a contradiction.
\end{proof}
By Claim \ref{strich} and construction, we obtain $d_{G_1}(X)=mn$. This yields $d_H(X \cap V_H)>d_{G_1}(X)-M=mn-M=k$, so $(H,k)$ is a positive instance of AMAXCUT.
\end{proof}
\subsection{The main construction}
We now construct an instance $(G_2,F)$ of CA. The graph $G_2=(V_2,E_2)$ is obtained from $G_1$ by replacing all vertices in $V_1-\{q,t\}$ by certain gadgets.

For every $v \in V_H$, $G_2$ contains an augmented $(M+m+1,m+\frac{d_H(v)}{2})$-grid $W^v$. Further, $G_2$ contains an augmented $(M+m+1,M+\frac{k}{2})$-grid $W^s$. Observe that $W^v$ for all $v \in V_H$ and $W^s$ are well-defined because $m,k,M$ and $d_H(v)$ for all $v \in V_H$ are even. Let $V_2=\cup_{v \in V_H}V(W^v)\cup V(W^s)\cup\{q,t\}$. We now add an edge from $q$ to each vertex in $L_{M}(W^s)$. We next add a perfect matching between $(L(W^s)-L_{M}(W^s))\cup P(W^s)$ and $\cup_{v \in V_H}L_{m}(W^v)$. Observe that this is possible because $|(L(W^s)-L_{M}(W^s))\cup P(W^s)|=\frac{k}{2}+M+\frac{k}{2}=mn=|\cup_{v \in V_H}L_{m}(W^v)|$. Finally, we add an edge from every vertex in $\cup_{v \in V_H}P_{m}(W^v)$ to $t$. Observe that $G_1$ can be obtained from $G_2$ by contracting each $W^v$ and $W^s$ into single vertices.

We now prove an important property of $G_2$.



\begin{lemma}\label{r}
For any $\emptyset \subset X \subset V_2$, we have\\
\begin{equation*} R_{G_2}(X) =2\lfloor\frac{\min\{\max\{d_{G_2}(v):v \in X\},\max\{d_{G_2}(v):v \in V_2-X\}\}}{2}\rfloor.
\end{equation*}
\end{lemma}
\begin{proof}
As $G_1$ is 4-edge-connected and by Lemma \ref{34conn} applied to $W^s$ and $W^v$ for all $v \in V_H$, we obtain that $\lambda_{G_2}(u,v)= \min\{d_{G_2}(u),d_{G_2}(v)\}$ for all $u,v \in V_2$ with $\{u,v\} \neq \{q,t\}$. This shows the statement for all $\emptyset \subset X \subset V_2$ such that $\{q,t\}\subseteq X$ or $\{q,t\}\subseteq V_2-X$. On the other hand, if $X$ is a $q \bar{t}$-set or a $t\bar{q}$-set, we have $\min\{\max\{d_{G_2}(v):v \in X\},\max\{d_{G_2}(v):v \in V_2-X\}\}=M$. As $M$ is even, it hence suffices to prove that $\lambda_{G_2}(q,t)=M$. 

We have $\lambda_{G_2}(q,t)\leq d_{G_2}(q)=M$. Next, there is an edge linking $q$ and $l_{i_1}(W^s)$ for all $i_1=1,\ldots ,M$ which can be concatenated to a path from $l_{i_1}(W^s)$ to $p_{i_1}(W^s)$ using only vertices of a single row of $W^s$. Now there is an edge from $p_{i_1}(W^s)$ to a vertex $l_{i_2}(W^v)$ for some $i_2\in\{1,\ldots,m\}$ and some $v \in V_H$. Finally, there is a path from  $l_{i_2}(W^v)$ to  $p_{i_2}(W^v)$ and an edge from $p_{i_2}(W^v)$ to $t$. This yields a set of $M$ edge-disjoint $qt$-paths, so $\lambda_{G_2}(q,t)\geq M$.

\end{proof}
For some $v \in V_H$, let  $B_v$ denote $(L(W^{v})-L_{m}(W^{v}))\cup (P(W^{v})-P_{m}(W^{v}))$. Now we define $F$ to be an odd-vertex pairing of $G_2$ in the following way: For every $uv \in E_H$, $F$ contains an edge linking $B_u$ and $B_v$. This is possible because for every $v \in V_H$, the set of vertices in $V(W^v)$ which are of odd degree in $G_2$ is exactly $B_v$ and $|B_v|=d_H(v)$.
\subsection{Reduction for CA}
This subsection is dedicated to proving the following lemma which gives a relation of the cut sizes in $G_1$ and $G_2$.
\begin{lemma}\label{ausbeul}
$(G_2,F)$ is a negative instance of CA if and only if there is some $q\bar{t}$-set $X\subseteq V_1$ such that $d_{G_1}(X)-d_H(X\cap V_H)<M$.
\end{lemma}
\begin{proof}
First suppose that there is some $q\bar{t}$-set $X\subseteq V_1$ such that $d_{G_1}(X)-d_H(X\cap V_H)<M$. Let $X'\subseteq V_2$ be the set that contains $q \cup \cup_{v \in X\cap V_H}V(W^v)$ and that contains $V(W^s)$ if $X$ contains $s$. Then Lemma \ref{r} yields $d_{G_2}(X')-d_F(X')=d_{G_1}(X)-d_H(X \cap V_H)<M = R_{G_2}(X')$, so $(G_2,F)$ is a negative instance of $AC$.
\medskip

Now suppose that $(G_2,F)$ is a negative instance of CA, so there is some $X\subset V_2$  such that $d_{G_2}(X)-d_F(X) < R_{G_2}(X).$  We choose $X$ among all such sets such that $d_{G_2}(X)$ is minimal.

\begin{claim}\label{yfgyi}
 Let $W\in W^s\cup\{W^v:v \in V_H\}.$ Then each connected component of $W[X]$ or $W[V_2-X]$  contains at least two vertices of $L(W)\cup P(W).$
\end{claim}
\begin{proof}
By symmetry and as $d_{G_2}(X) = d_{G_2}(V_2-X),$ it suffices to prove the statement for $W[X].$ For the sake of a contradiction, suppose that for the vertex set $C$ of a connected component of $W[X]$, we have $|C\cap(L(W)\cup P(W))|\le 1.$ 

First suppose that $X=C$. If $X$ consists of a single vertex $v$ with $d_F(v)=1$, Lemma \ref{r} yields $d_{G_2}(X)-d_F(X) =3-1=2= R_{G_2}(X),$ a contradiction. Otherwise, Lemma \ref{34conn} yields $d_{G_2}(X)\geq 4$ and so, as $d_F(X)\leq 1$ and $G+F$ is eulerian, we obtain by Lemma \ref{r} that $d_{G_2}(X)-d_F(X) \geq 4= R_{G_2}(X),$ a contradiction.

We may hence suppose that $X'=X-C$ is nonempty, so, by Lemma \ref{r} and as $q,t \notin V(W)$, we have $ R_{G_2}(X)- R_{G_2}(X')\leq 4-2=2$. If $C$ consists of a single vertex $v$ with $d_F(v)=0$, we obtain $d_{G_2}(X')-d_F(X')\leq d_{G_2}(X)-2-d_F(X)< R_{G_2}(X)-2\leq R_{G_2}(X')$, a contradiction to the minimality of $X$. Otherwise, Lemma \ref{34conn} yields $d_{G_2}(X)-d_{G_2}(X')\geq d_W(X)-1\geq 4-1=3$ and $d_F(X')-d_F(X)\leq 1$. This yields $d_{G_2}(X')-d_F(X')\leq (d_{G_2}(X)-3)-(d_F(X)-1)< R_{G_2}(X)-2\leq R_{G_2}(X')$, a contradiction to the minimality of $X$.
\end{proof}

We are now ready to show that $V(W) \subseteq X$ or $V(W) \cap X \neq \emptyset$ for every $W\in W^s\cup\{W^v:v \in V_H\}$. Suppose otherwise, then by Claim \ref{yfgyi}, both $W[X]$ and $W[V_2-X]$ have a connected component each containing at least two vertices of $L(W)\cup P(W)$. By Lemmas \ref{big} and \ref{r}, this yields $d_{G_2}(X)-d_F(X)\geq M+m+1-m> M \geq R_{G'}(X)$, a contradiction.

Now let $X^*\subseteq V_1$ be the set of vertices that contains $v$ whenever $V(W^v) \subseteq X$ and $s$ if $V(W^s) \subseteq X$. Observe that $d_{G_2}(X)= d_{G_1}(X^*)\geq 2m$ by construction. Also, observe that $d_F(X)=d_H(X^*\cap V_H)$. By symmetry, we may suppose that $q \in X$. If $X$ is not a $q\bar{t}$-set, Lemma \ref{r} yields $d_{G_2}(X)-d_F(X)\geq d_{G_1}(X^*)-m\geq 2m-m=m>4 \geq R_{G_2}(X)$, a contradiction. If $X^*$ is a $q\bar{t}$-set, by Lemma \ref{r}, we obtain $d_{G_1}(X^*)-d_H(X^*\cap V_H)=d_{G_2}(X)-d_F(X)<R_{G_2}(X)=M$. 
\end{proof}
 \subsection{Reduction for OA}

The following result can be obtained by analogous methods to the proof of Lemma \ref{ausbeul}.

\begin{lemma}\label{notbig}
There is no $X \subseteq V_2$ such that $d_{G_2}(X)<d_F(X)$.
\end{lemma} 

We here prove the following result that allows for a reduction for OA. While this proof does not require any new arguments apart from Lemma \ref{notbig}, we include it here for the sake of self-containment.
The first implication is part of the proof of Nash-Williams of Theorem \ref{well} in \cite{NW} while the arguments for the second implication can be found in a similar form in \cite{ks}.

\begin{lemma}\label{equiv}
$(G_2,F)$ is a negative instance of $OA$ if and only if $(G_2,F)$ is a negative instance of $CA$.
\end{lemma}
\begin{proof}

 First suppose that $(G_2,F)$ is a negative instance of OA. Then there is an eulerian orientation $\vec{G_2}+\vec{F}$ of $G_2+F$ such that $\vec{G_2}$ is not well-balanced. This means that there are some  $u,v \in V_2$ such that $\lambda_{\vec{G_2}}(u,v)<\lfloor\frac{\lambda_{G_2}(u,v)}{2}\rfloor$. Therefore there is some $u\bar{v}$-set $X \subset V_2$ such that $d_{\vec{G_2}}^+(X)<\lfloor\frac{\lambda_{G_2}(u,v)}{2}\rfloor$. As $G_2+F$ is eulerian, we obtain that $d_F(X)\geq d_{\vec{G_2}}^-(X)-d_{\vec{G_2}}^+(X)=d_{G_2}(X)-2d_{\vec{G_2}}^+(X)>d_{G_2}(X)-2\lfloor\frac{\lambda_{G_2}(u,v)}{2}\rfloor\geq d_{G_2}(X)-R_{G_2}(X)$, so $(G_2,F)$ is a negative instance of CA.

For the other direction, suppose that $(G_2,F)$ is a negative instance of CA, so there is some $X\subset V_2$ such that  $d_{G_2}(X)-d_F(X)<R_{G_2}(X)$. Let $u \in X$ and $v \in V_2-X$ such that $R_{G_2}(X)=2\lfloor\frac{\lambda_{G_2}(u,v)}{2}\rfloor$. Let $\vec{F}$ be an orientation of $F$ such that all the edges with exactly one endvertex in $X$ are directed away from $X$. By Lemma \ref{notbig} and Theorem \ref{ff}, there is an orientation $\vec{G_2}$ of $G_2$ such that $\vec{G_2}+\vec{F}$ is eulerian. This yields $\lambda_{\vec{G_2}}(u,v)\leq d^+_{\vec{G_2}}(X)=\frac{1}{2}(d_{G_2}(X)+d_F(X))-d^+_{\vec{F}}(X)=\frac{1}{2}(d_{G_2}(X)+d_F(X))-d_{F}(X)=\frac{1}{2}(d_{G_2}(X)-d_F(X))<\frac{1}{2}R_{G_2}(X)=\lfloor\frac{\lambda_{G_2}(u,v)}{2}\rfloor$. We obtain that $\vec{G_2}$ is not well-balanced, so $(G_2,F)$ is a negative instance of OA.
\end{proof}
\subsection{Conclusion}

By Lemmas \ref{int} and \ref{ausbeul}, we obtain that $(G_2,F)$ is a negative instance of $CA$ if and only if $(H,k)$ is a positive instance of AMAXCUT. By Lemma \ref{amax} and as the size of $(G_2,F)$ is polynomial in the size of $(H,k)$, we obtain Theorem \ref{ca}.

By Lemmas \ref{int}, \ref{ausbeul} and \ref{equiv}, we obtain that $(G_2,F)$ is a negative instance of $OA$ if and only if $(H,k)$ is a positive instance of AMAXCUT. By Lemma \ref{amax} and as the size of $(G_2,F)$ is polynomial in the size of $(H,k)$, we obtain Theorem \ref{oa}.

\section{Local arc-connectivity orientation}\label{local}
This section is dedicated to proving Theorem \ref{loc}. We need to consider the following algorithmic problem.

\medskip
{\it Bounded well-balanced orientation (BWBO)}

\smallskip
{\bf Instance} A graph $G=(V, E)$ and two functions $l^+,l^-:V \rightarrow \mathbb{Z}_{\geq 0}$.

\smallskip
{\bf Question} Is there a well-balanced orientation $\vec{G}$ of $G$ such that $d_{\vec{G}}^+(v)\geq l^+(v)$ and $d_{\vec{G}}^-(v)\geq l^-(v)$ for all $v \in V$?
\medskip

The following result is proven in \cite{bikks}.
\begin{lemma}
BWBO is NP-hard.
\end{lemma}

We are now ready to give the reduction for Theorem \ref{loc}.

\begin{proof}(of Theorem \ref{loc})

We prove this by a reduction from BWBO. Let $(G=(V,E),l^+,l^-)$ be an instance of BWBO. We add two vertices $x$ and $y$ and for every $v \in V$, we add $d_G(v)$ edges linking $v$ and each of $x$ and $y$. We denote this graph by $G'=(V',E')$. Observe that $|V'|=|V|+2$ and $|E'|=5|E|$, so the size of $G'$ is polynomial in the size of $G$. We now define $r:(V')^2 \rightarrow \mathbb{Z}_{\geq 0}$ by $r(u,v)=\lfloor\frac{\lambda_G(u,v)}{2}\rfloor$ for all $u,v \in V^2$, $r(x,v)=d_G(v)+l^-(v),r(v,x)=0,r(y,v)=0$ and $r(v,y)=d_G(v)+l^+(v)$ for all $v \in V$ and $r(x,y)=2|E|$.

We prove that $(G',r)$ is a positive instance of LACO if and only if $(G,l^+,l^-)$ is a positive instance of BWBO. First suppose that $(G',r)$ is a positive instance of LACO, so there is an orientation $\vec{G'}$ of $G'$ such that $\lambda_{\vec{G'}}(u,v)\geq r(u,v)$ for all $u,v \in (V')^2$. Observe that $d_{G'}(x)=r(x,y)=d_{G'}(y)$, so $x$ is a source and $y$ is a sink in $\vec{G'}$. We show that $\vec{G}$, the restriction of $\vec{G'}$ to $G$, is a well-balanced orientation of $G$ such that $d_{\vec{G}}^+(v)\geq l^+(v)$ and $d_{\vec{G}}^-(v)\geq l^-(v)$ for all $v \in V$. As $x$ is a source and $y$ is a sink in $\vec{G'}$, for any $u,v \in V^2$, we have $\lambda_{\vec{G}}(u,v)=\lambda_{\vec{G'}}(u,v)\geq r(u,v)=\lfloor\frac{\lambda_G(u,v)}{2}\rfloor$, so $\vec{G}$ is well-balanced. Further, for any $v \in V$, we have $d_{\vec{G}}^-(v)=d_{\vec{G'}}^-(v)-d_{\vec{G'}}(x,v)\geq \lambda_{\vec{G'}}(x,v)-d_{\vec{G'}}(x,v)\geq r(x,v)-d_{\vec{G'}}(x,v)=d_G(v)+l^-(v)-d_G(v)=l^-(v)$. Similarly, $d_{\vec{G}}^+(v)\geq l^+(v)$, so $(G,l^+,l^-)$ is a positive instance of $BWBO$.

Now suppose that  $(G,l^+,l^-)$ is a positive instance of $BWBO$, so there is a well-balanced orientation $\vec{G}$ of $G$ such that $d_{\vec{G}}^+(v)\geq l^+(v)$ and $d_{\vec{G}}^-(v)\geq l^-(v)$ for all $v \in V$. We complete this to an orientation $\vec{G'}$ of $G'$ by orienting all edges incident to $x$ away from $x$ and  all edges incident to $y$ toward $y$. As $\vec{G}$ is well-balanced, we have $\lambda_{\vec{G'}}(u,v)=\lambda_{\vec{G}}(u,v)\geq \lfloor\frac{\lambda_G(u,v)}{2}\rfloor=r(u,v)$ for all $u,v \in V^2$. By construction, we have $\lambda_{\vec{G'}}(x,y)=\sum_{v \in V}d_G(v)=2|E|=r(x,y)$. For any $v \in V$, we have $d_G(v)$ arc-disjoint $xv$-paths of length $1$. Further, for every arc $uv$ entering $v$ in $\vec{G}$, we have a path $xuv$. As all these paths can be chosen to be arc-disjoint, we obtain that $\lambda_{\vec{G'}}(x,v)\geq d_{\vec{G'}}(x,v)+d^-_{\vec{G}}(v)\geq d_G(v)+l^-(v)=r(x,v)$. Similarly, $\lambda_{\vec{G'}}(v,y)\geq r(v,y)$, so $(G',r)$ is a positive instance of LACO.
\end{proof}
\section*{Acknowledgement}
I wish to thank Zolt\'an Szigeti. He first suggested the problems. Later, he carefully proofread the article and proposed some simplifications.

\end{document}